\documentclass[a4paper]{article}

\usepackage{amsfonts}
\usepackage{amsopn}
\usepackage{amsmath}
\usepackage{amsthm}
\usepackage{latexsym}
\usepackage{amssymb}

\newcommand\R{\mathbb R}
\newcommand\N{\mathbb N}
\newcommand\cB{{\mathcal B}}
\newcommand\cD{{\mathcal D}}
\newcommand\cS{{\mathcal S}}
\newcommand\cP{{\mathcal P}}
\newcommand{\supp}{\operatorname{supp}}
\newcommand{\pr}{\operatorname{pr}}

\newcommand{\Graph}{\operatorname{Graph}}
\newcommand{\clos}{\operatorname{cl}}
\newcommand{\cc}{\operatorname{cc}}
\newcommand\comp{^{\sf c}}

\newtheorem{theorem}{Theorem}[section]
\newtheorem{lemma}[theorem]{Lemma}
\newtheorem{remark}[theorem]{Remark}
\newtheorem{corollary}[theorem]{Corollary}
\newtheorem{proposition}[theorem]{Proposition}

\title{Random Set Solutions to Stochastic Wave Equations}

\author{Michael Oberguggenberger\thanks{Unit of Engineering Mathematics, University of Innsbruck,
Technikerstra\ss e 13, 6020 Innsbruck,
Austria, (michael.oberguggenberger@uibk.ac.at)}
 \and
Lukas Wurzer\thanks{Johannesplatz 1, 77815 B\"{u}hl, Germany
(wurzer.lukas@gmail.com)}
}
\date{}

\begin{document}
\maketitle

\begin{abstract}
  This paper is devoted to three topics. First, proving a measurability theorem for multifunctions with values in non-metrizable spaces, which is required to show that solutions to stochastic wave equations with interval parameters are random sets; second, to apply the theorem to wave equations in arbitrary space dimensions; and third, to computing upper and lower probabilities of the values of the solution in the case of one space dimension.
\end{abstract}

\section{Introduction}\label{sec:intro}
Let $\left(\Omega,\Sigma,P\right)$ be a probability space, $S$ a topological space, $\mathcal P\left(S\right)$ the power set of $S$ and $\mathcal B\left(S\right)$
the Borel $\sigma$-algebra on $S$. A multifunction $X:\Omega\to\mathcal P\left(S\right)$ is a \emph{random set}, if the upper inverses
\begin{equation}\label{borel}
X^-(B) = \{\omega\in\Omega:X(\omega)\cap B\neq\emptyset\}
\end{equation}
are measurable for every Borel set $B\in \mathcal B\left(S\right)$. Consider the stochastic wave equation
\begin{equation}
\label{swec}
\left(\partial_t^2-c^2\Delta\right)u_c=\dot W
\end{equation}
on $\R^{d+1}$, where $\dot W$ is space-time Gaussian white noise, $\Delta = \partial_{x_1}^2 + \cdots + \partial_{x_d}^2$ is the Laplacian. It is well-known \cite{Walsh} that in space dimensions $d\geq 2$, the solution $u_c$ is a generalized stochastic process, that is, its trajectories $u_c(\omega)$ belong to the Schwartz space of distributions
$\mathcal D'(\R^{d+1})$. (For all notions from the theory of distributions we refer to \cite{hor}.)
Replacing the propagation speed $c$ by an interval $[\underline c,\overline c]$ with $\underline c > 0$ results in a multifunction
\begin{equation}\label{swrs}
    X(\omega) = \left\{u_c(\omega):\underline c\leq c\leq\overline c\right\}
\end{equation}
with values in the power set of $\mathcal D'(\R^{d+1})$. A Castaing representation of this multifunction is readily obtained from the continuous dependence of $u_c$ on its parameter $c$. Endowing $\mathcal D'(\R^{d+1})$ with the weak-$\ast$ topology, one would like to show that $X$ is a random set. However, $\mathcal D'(\R^{d+1})$ is not metrizable
and so the Fundamental Measurability Theorem \cite{mol} for Polish spaces does not apply.

The first part of the paper will be devoted to proving a generalization of the Fundamental Measurability Theorem, showing that a sequentially compact, Effros measurable  multifunction (see below) with values in the dual of a separable locally convex space, which is itself separable and a Souslin space, is a random set.
The second part of the paper provides the details required to apply this result to the wave equation in arbitrary space dimensions, culminating in the proof that the multifunction $X$ in \eqref{swrs} is a random set.
In the third part we will focus on the stochastic wave equation in one space dimension ($d=1$). In this case, the solution is a modified Brownian sheet and consequently it has almost surely continuous trajectories. Thus one can consider the $\R$-valued multifunctions
\begin{equation}\label{swrs1d}
    X(x,t,\omega) = \left\{u_c(x,t,\omega):\underline c\leq c\leq\overline c\right\}
\end{equation}
at fixed $(x,t)$. It turns out that the maps
$(r,\omega) \to \frac2t\, u_{1/r}(x,t,\omega)$
are Brownian motions with respect to $r \in [0,\infty)$. Using well-known formulas for the first hitting times, this will allow us to compute the upper and lower distribution functions of the random sets \eqref{swrs1d}.

The presented results have been obtained by the second author in \cite{Wurzer}. Various other generalizations of the Fundamental Measurability Theorem beyond Polish spaces can be found in \cite{averna, castaing, Himmelberg}.

\section{Measurability of Multifunctions}
\label{sec:meas}
As suggested by the defining property, a random set is also referred to as a \emph{Borel measurable multifunction}. Various measurability properties will play a role in the Fundamental Measurability Theorem. A multifunction $X$ is called \emph{Effros measurable} if its upper inverses in \eqref{borel} are measurable for every open set $B\subset S$. The multifunction $X$ is called \emph{graph measurable} if its graph
\begin{equation*}
    \Graph\left(X\right)=\left\{(\omega,x)\in\Omega\times S:x\in X(\omega)\right\}
\end{equation*}
is measurable, that is, belongs to $\Sigma\otimes\mathcal B\left(S\right)$. A random variable $\xi$ with values in $S$ is called a selection of $X$ if $\xi(\omega) \in X(\omega)$ for almost all $\omega\in\Omega$. Let $X$ be a closed-valued random set. A countable family of selections $\xi_n$ is said
to be a \emph{Castaing representation} of $X$ if
\begin{equation}
  X(\omega)=\clos\left\{\xi_n(\omega),n\geq 1\right\}\label{sel}
\end{equation}
for (almost) all $\omega\in\Omega$, where $\clos$ denotes the closure in $S$.
\begin{remark}
\label{rem1}
If a multifunction is Effros measurable, then its lower inverses
\begin{equation}\label{lowinv}
   X_-(B) = \{\omega\in\Omega: X(\omega) \subset B\}
\end{equation}
are measurable for every closed set $B\subset X$, and vice versa. This follows from the formula $X^-(B\,\comp) = (X_-(B))\comp$.

(b) If a closed valued multifunction $X:\Omega\to S$ into some topological space $S$ has a Castaing representation, then it is Effros measurable. Indeed, let $\left(\xi_n\right)_{n\in\N}$ be a Castaing representation of $X$ and let $B\subset S$ open. On the one hand, $\xi_n(\omega)\in X(\omega)$ for (almost) all $\omega\in\Omega$ and all $n\in\N$. Thus $\left\{\omega:\xi_n\left(\omega\right)\in B\right\}\subset X^-\left(B\right)$ for all $n\in\N$. On the other hand, let
$\omega\in X^-(B)$. Since $B$ is open, \eqref{sel} implies the existence of an $n\in\N$ such that $\xi_n(\omega)\in B$ for (almost) all $\omega\in\Omega$. Thus
\begin{equation*}
    X^-\left(B\right)=\bigcup_{n\in\N}\left\{\omega:\xi_n\left(\omega\right)\in B\right\}.
  \end{equation*}
Measurability of $X^-(B)$ follows from measurability of the selections $\xi_n$.

(c) When talking about multifunctions $X:\Omega\to\mathcal P\left(S\right)$ we assume from now on that
\[
   X(\omega) \neq \emptyset \quad \mbox{for (almost)\ all\ }\omega\in\Omega.
\]
\end{remark}
Finally, a \emph{Polish space} is a separable, metrizable and complete topological space, while a \emph{Souslin space} is a Hausdorff space which is the continuous image of a Polish space.
Let $\left(\Omega,\Sigma,P\right)$ be a complete probability space, $S$ a Polish space, and $X:\Omega\to\mathcal P\left(S\right)$ a closed-valued multifunction. The Fundamental Measurability Theorem \cite{mol} says that $X$ is Borel measurable if and only if it is Effros measurable if and only if it is graph measurable if and only if it admits a Castaing representation.

In what follows, $E$ is a Hausdorff locally convex topological vector space and $E'$ its continuous dual, equipped with the weak-$\ast$ topology. We denote by
$\langle x,e\rangle$ the action of $x\in E'$ on $e\in E$. The probability space $(\Omega, \Sigma, P)$ is assumed to be complete throughout. Here is the main result of this section.
\begin{theorem}\label{thm:fund}
Assume that both $E$ and $E'$ are separable and that $E'$ is Souslin. Let $X:\Omega\to\mathcal P\left(E'\right)$ be an Effros measurable multifunction with sequentially compact values. Then $X$ is Borel measurable, i.\,e., a random set.
\end{theorem}
The proof requires a few preparations. Let $(e_i)_{i\in\N}\subset E$ and $(x_i)_{i\in\N}\subset E'$ be dense sequences. Given $x\in E'$, the sets
\[
  U_{mn}(x)=\left\{y\in E' : \left|\left<y-x,e_1\right>\right|<\tfrac 1 m,\dots,\left|\left<y-x, e_n\right>\right|<\tfrac 1 m\right\}
\]
with $m,n\in\N$, form a countable set of neighborhoods of $x$. In the following lemmas, $E'$ is not required to be a Souslin space.
\begin{lemma}\label{PR}
  Let $A\subset E'$ be sequentially compact and $x_0\in E'$ such
  that $x_0\notin A$. Then there exists an $n\in\N$ such that
  $U_{1n}(x_0)\cap A=\emptyset$.
\end{lemma}
\begin{proof}
  We assume the converse and derive a contradiction. So suppose that $x_0\notin A$ and
  $U_{1n}(x_0)\cap A\neq\emptyset$ for all $n\in\N$. Then one can find
  a sequence $(x_n)_{n\in\N}$ such that $x_n\in U_{1n}(x_0)\cap A$
  for all $n\in\N$. Since $(x_n)_{n\in\N}$ is a subset of $A$, which
  is sequentially compact, we can introduce a convergent subsequence by
  $(n_k)_{k\in\N}\subset\N$, such that
  \begin{equation*}
    \label{xbar}
    \bar x =\lim_{k\to\infty}x_{n_k}\in A.
  \end{equation*}
  From $x_0\notin A$ it follows that $\bar x\neq x_0$, and so there is  $e\in E$
  such that
  \begin{equation}\label{g2}
    \left|\left<x_0-\bar x,e\right>\right|\geq2.
  \end{equation}
  On the other hand, since $(e_m)_{m\in\N}$ is dense in $E$,
  there exists another sequence $(m_k)_{k\in\N}$, such that
  \begin{equation*}
    \lim_{k\to\infty}\left|\left<x_0-\bar x,e-e_{m_k}\right>\right|=0.
  \end{equation*}
  This implies that there exists an $m\in\N$ such that
  \begin{equation}
    \label{l1}
    \left|\left<x_0-\bar x,e-e_m\right>\right|<1.
  \end{equation}
  Equations \eqref{g2} and \eqref{l1} lead to
  \begin{align}
    \label{g1}
    \left|\left<x_0-\bar x,e_m\right>\right|&=\left|\left<x_0-\bar x,e\right>-\left<x_0-\bar x,e-e_m\right>\right| \nonumber \\
     &\geq\left|\left<x_0-\bar
        x,e\right>\right|-\left|\left<x_0-\bar
        x,e-e_m\right>\right|>1.
  \end{align}
  Combining \eqref{g1} and \eqref{g2} gives
  \begin{align*}
    \lim_{k\to\infty}\left|\left<x_0-x_{n_k},e_m\right>\right|&=\big|\left<x_0-\bar x,e_m\right>+\lim_{k\to\infty}\left<\bar x-x_{n_k},e_m\right>\big|\\
    &=\left|\left<x_0-\bar x,e_m\right>\right|>1.
  \end{align*}
  Consequently, there is $l\in\N$ such that
  \begin{equation}
    \label{fk}
    \left|\left<x_0-x_{n_k},e_m\right>\right|>1,\quad \mbox{for\ all\ } k>l.
  \end{equation}
  On the other hand, we have assumed that $x_{n_k}\in U_{1n_k}(x_0)$ for all $k\in\N$, i.\,e.
  \begin{equation*}
    \left|\left<x_{n_k}-x_0,e_1\right>\right|<1,\ \dots\ ,\ \left|\left<x_{n_k}-x_0,e_{n_k}\right>\right|<1.
  \end{equation*}
  But this contradicts \eqref{fk}, if $k$ is large enough such that $n_k>m$.
\end{proof}
\begin{lemma}
  \label{le}
  Let $x_0\in E'$. Then there is an $l\in\N$
  such that
  \begin{align}
      x_0&\in U_{2n}\left(x_l\right),\label{x0u2nxl}\\
      U_{2n}(x_l)&\subset U_{1n}(x_0)\label{u2nxlu1nx0}.
  \end{align}
  \end{lemma}
\begin{proof}
  We choose $l$ in a way that
  \begin{equation}
    \label{xlu2nx0}
    x_l\in U_{2n}(x_0),
  \end{equation}
  which means that
  \begin{equation*}
    \left|\left\langle x_l-x_0,e_k\right\rangle\right|<\tfrac{1}{2}\quad \mbox{for\ all\ } k\in\{1,2,\dots,n\}.
  \end{equation*}
  This is possible, since $(x_l)_{l\in\N}$ is dense in $E'$. Then
  \eqref{x0u2nxl} follows immediately from \eqref{xlu2nx0}. Now
  suppose that $y\in U_{2n}(x_l)$, i.\,e.
  \begin{equation*}
    \left|\left<y-x_l,e_k\right>\right|<\tfrac{1}{2}\quad \mbox{for\ all\ } \in\{1,2,\dots,n\}.
  \end{equation*}
  Then we get for all $k\in\{1,2,\dots,n\}$ that
  \begin{equation*}
    \left|\left<y-x_0,e_k\right>\right|\leq\left|\left<y-x_l,e_k\right>\right|+\left|\left<x_l-x_0,e_k\right>\right|<\tfrac{1}{2}+\tfrac{1}{2}=1,
  \end{equation*}
  and so $y\in U_{1n}(x_0)$. This implies \eqref{u2nxlu1nx0}.
\end{proof}
In order to simplify our notation we define a sequence
$\left(V_i\right)_{i\in\N}$ of open subsets of $E'$ such that
\begin{equation*}
  \left\{V_1,V_2,\dots\right\}=\left\{U_{2n}\left(x_l\right):n,l\in\N\right\}.
\end{equation*}
Combining Lemmas~\ref{PR} and \ref{le} leads to the following
\begin{corollary}
  \label{sep}
  For any sequentially compact set $A\subset E'$ and any
  $x_0\in E'$ with $x_0\notin A$, there exists an index $j\in\N$
  such that $x_0\in V_j$ and $V_j\cap A=\emptyset$.
\end{corollary}
\begin{lemma}
  \label{GX}
  Assume that both $E$ and $E'$ are separable. Let $X:\Omega\to\mathcal P\left(E'\right)$ be an Effros measurable multifunction with sequentially compact values. Then $X$ is graph measurable, i.\,e., $\Graph\left(X\right)$ is $\Sigma\otimes\cB(E')$-measurable.
\end{lemma}
\begin{proof}
Recall that the graph of $X$ is
  defined by
  \begin{equation*}
    \Graph(X)=\left\{(\omega,x)\in\Omega\times E',x\in  X(\omega)\right\}.
  \end{equation*}
  We will show that
  \begin{equation*}
    \Graph\left( X\right)\comp=\bigcup_{n\in\N} X_-\left(V_n\comp\right)\times V_n.
  \end{equation*}
  By Remark\;\ref{rem1}(a), this implies the measurability of $ \Graph(X)$. Assume first that
  \begin{equation*}
    (\omega,x)\in\Graph\left( X\right)\comp.
  \end{equation*}
  It follows that $(\omega,x)\notin\Graph\left( X\right)$, and therefore
  $x\notin  X(\omega)$. By assumption, $ X(\omega)$ is
  sequentially compact, and we can apply
  Corollary~\ref{sep}. Consequently, there is $j\in\N$ such that
  $V_j\cap  X(\omega)=\emptyset$ and $x\in V_j$. Hence,
  $\omega\notin  X^-(V_j)$, or equivalently,
  $\omega\in  X_-(V_j\comp)$. We get $(\omega,x)\in
   X_-(V_j\comp)\times V_j$ and it follows that
  \begin{equation}
    \label{omx}
    (\omega,x)\in\bigcup_{n\in\N} X_-(V_n\comp)\times V_n.
  \end{equation}
  Conversely, assume that \eqref{omx} holds. Then there is $j\in\N$
  such that $(\omega,x)\in  X_-(V_j\comp)\times
  V_j$. We follow the arguments above in the opposite direction and
  get $(\omega,x)\in\Graph\left( X\right)\comp$.
\end{proof}
The final ingredient in the proof of Theorem\;\ref{thm:fund}
is the \emph{projection theorem} \cite[Theorem III.23]{castaing}.
\begin{proposition}\label{pro}
  Let $(\Omega,\Sigma,P)$ be a complete probability space, $S$ a Souslin space and $H$ a subset of $\Omega\times S$. If $H$ belongs to
  $\Sigma\otimes\cB(S)$, its projection $\pr_\Omega(H)$ belongs to $\Sigma$.
\end{proposition}
\begin{proof}\textbf{of Theorem{thm:fund}.}
By Lemma \;\ref{GX}, the graph of $X$ is measurable, i.\,e., belongs to $\Sigma\otimes\cB(E')$.
Let $B\in\cB(E')$. Then
\begin{equation*}
\Graph\left(X\right)\cap\left(\Omega\times B\right)\in\Sigma\otimes\cB(E').
\end{equation*}
Observe that
\begin{equation*}
X^-\left(B\right) = \pr_\Omega\big(\Graph(X)\cap(\Omega\times B)\big).
\end{equation*}
The projection theorem implies that $X^-\left(B\right)\in\Sigma$.
\end{proof}

\section{Application to the Wave Equation}
\label{sec:GenSol}
We address generalized solutions in $\cD'(\R^{d+1})$ to the half-space problem for the wave equation
\begin{equation}
\label{wave}
\left(\partial_t^2-c^2\Delta\right)u_c = g,\quad u_c\big|_{t < 0} = 0
\end{equation}
In other words, the support $\supp u_c$ of the solution should be a subset of $\R^d\times [0,\infty)$. Here $g\in\cD'(\R^{d+1})$ with
$\supp g \subset\R^d\times [0,\infty)$ and $c$ is a positive constant.

\subsection{The Deterministic Case}
\label{subsec:det}
It is wellknown \cite{Chazarain1982} that the problem\;\eqref{wave} has a unique solution $u_c \in \cD'(\R^{d+1})$. In particular, for $g = \delta$,
the Dirac measure at zero in $\R^{d+1}$, one obtains the unique \emph{fundamental solution} $F_c$ with support in the forward light cone
$\{(x,t)\in \R^{d+1}:t\geq 0, |x|\leq ct\}$. If $\supp g \subset\R^d\times [0,\infty)$, the convolution of $F_c$ and $g$ exists, and the solution $u_c$
to \eqref{wave} is given by $u_c = F_c\ast g$. Explicit formulas for the fundamental solution can be found, e.\,g., in \cite{TrevesBasic}. In space dimensions $d=1$ and $d=2$, it is a function, in space dimension $d=3$, it is a measure, and in space dimensions $d\geq 4$, it is a distribution of higher order.

Recall that the Dirac measure in $\R^{d+1}$ is homogeneous of degree $-1$ with respect to $t$, as follows from the calculation
\begin{align*}
   \langle \delta(x,ct),\varphi(x,t)\rangle &= \langle \delta(x,t), \tfrac1c\varphi(x,\tfrac{t}{c})\rangle\\
       &= \tfrac1c\varphi (0,0) = \tfrac1c\langle \delta(x,t),\varphi(x,t)\rangle
\end{align*}
for test functions $\varphi\in\cD(\R^{d+1})$. Here we have used the definition of homothety for distributions. (We are adhering to the common abusive notation exhibiting the variables $(x,t)$ in the duality bracket.) A similar calculation shows that
\[
   F_c(x,t) = \tfrac1c F_1(x,ct).
\]
\begin{lemma}\label{lem:continuity}
The maps $(0,\infty)\to \cD'(\R^{d+1})$, $c\to F_c$ and $c\to u_c$ are continuous.
\end{lemma}
\begin{proof}
For the first assertion, observe that the map $(0,\infty)\to \cD(\R^{d+1})$, $c\to\frac1{c^2}\varphi(x,\frac{t}{c})$ is continuos. Thus
\[
  c\to \langle \tfrac1c F_1(x,ct),\varphi(x,t)\rangle = \langle F_1(x,t), \tfrac1{c^2}\varphi(x,\tfrac{t}{c})\rangle
\]
is continuous as well. Actually, on can say more: If $c_0 > 0$, the map $[c_0,\infty) \to \cD'_\Gamma(\R^{d+1})$, $c\to F_c$, where
$\cD'_\Gamma(\R^{d+1})$ denotes the space of distributions with support in the cone $\Gamma = \{(x,t)\in \R^{d+1}:t\geq 0, |x|\leq c_0t\}$, is continuous. To prove the second assertions, it suffices to recall that for fixed $g\in \cD'(\R^{d+1})$ with support in the upper half-space, convolution $f\to f\ast g$ is a continuous map from
$\cD'_\Gamma(\R^{d+1})$ to $\cD'(\R^{d+1})$, see e.\,g. \cite[Section 4]{Vladimirov} or the arguments in \cite[Section 4.1]{Wurzer}.
\end{proof}

\subsection{The Stochastic Case}
\label{subsec:stoch}
Let $\left(\Omega,\Sigma,P\right)$ be a probability space and
$d\in\N$. A \emph{generalized stochastic process} $X$ is a weakly
measurable map
\begin{equation*}
  X:\Omega\to\cD'(\R^{d+1}),
\end{equation*}
i.\,e. $\omega\to\left<X(\omega),\varphi\right>$ is measurable
for all $\varphi\in\cD'(\R^{d+1})$.  An important example is
\emph{Gaussian space-time white noise}, denoted by $\dot W$, with support in
$D=\R^d\times\left[0,\infty\right)$. Let $\cS(D) = \cS(\R^{d+1})|D$, where $\cS(\R^{d+1})$ is the Schwartz space of rapidly decreasing smooth functions.
Its continuous dual $\Omega=\cS'(D)$, equipped with the Borel $\sigma$-algebra $\Sigma = \cB(\cS'(D))$ with respect to the weak-$\ast$ topology, is a measurable space.
By the \emph{Bochner-Minlos theorem} \cite[Section~3.2]{hida}, \cite[Theorem~2.1.1]{holden}.
there exists a unique probability measure $P$ on $\cS'(D)$ such that
  \begin{equation}\label{eq:wnp}
    \int_{\cS'(D)}
    {\rm e}^{{\rm i}\left\langle\omega,\varphi\right\rangle}dP(\omega)={\rm e}^{-\frac{1}{2}\left\Vert\varphi\right\Vert^2_{L^2(D)}}
  \end{equation}
for all $\varphi\in \cS(D)$. The completion of $(\Omega,\Sigma,P)$ is called the \emph{white noise probability space}. Space-time white noise $\dot W$ with support in $D$ is the generalized stochastic process
\begin{equation}
\label{noise}
  \begin{aligned}
    \dot W:\cS'(D)&\to\cD'(\R^{d+1}),\\
    \left\langle\dot
      W(\omega),\varphi\right\rangle&=\left\langle\omega,\varphi\vert D\right\rangle,
  \end{aligned}
\end{equation}
for all $\varphi\in\cD(\R^{d+1})$. By the defining equation \eqref{eq:wnp}, it is a Gaussian process. Its probabilistic properties are characterized by the It\^{o} isometry
\[
   {\rm E}\langle \dot W,\varphi\rangle = 0,\quad {\rm E}(\langle \dot W,\varphi\rangle)^2 = \|\varphi\|^2_{L^2(D)}.
\]
\begin{proposition}
\label{usswe}
There is an (almost surely) unique generalized stochastic process $u_c$ such that
\begin{align}
      \label{stwe1}\left(\partial_t^2-\Delta\right)u_c(\omega)&=\dot W(\omega),\\
      \label{stwe2}\supp u(\omega)&\subset\R^d\times[0,\infty),
\end{align}
in $\cD'\left(\R^{d+1}\right)$ for (almost) all $\omega\in\Omega$. Further $u_c(\omega)$ depends continuously on $c\in (0,\infty)$.
\end{proposition}
\begin{proof}
Pathwise existence and uniqueness follows from the deterministic theory outlined in Subsection\;\ref{subsec:det}. The solution is given by
\begin{equation*}
  u_c(\omega)=F_c\ast\dot W(\omega)
\end{equation*}
and depends continuously on $c$ by Lemma\;\ref{lem:continuity}.
Measurability at fixed $c$ follows from the fact that $u_c$ is obtained as a composition of a measurable and a continuous map:
\[
   \omega \to \dot{W}(\omega)\to F_c\ast\dot W(\omega),
\]
see again Lemma\;\ref{lem:continuity}.
\end{proof}
Modelling the uncertainty in the propagation speed by a compact interval $[\underline c, \overline c]$ with $\underline c > 0$, we are now in the position to define the solution to the stochastic wave equation as a multifunction
\begin{align}\label{shx}
    X:\Omega &\to\cP\big(\mathcal D'(\mathbb R^{d+1})\big)\nonumber\\
    X(\omega)&=\left\{u_c\left(\omega\right):\underline c\leq c\leq\overline c\right\}.
\end{align}
As continuous images of a compact interval, the values $X(\omega)$ are sequentially compact and, in particular, closed.
\begin{lemma}\label{lem:Effros}
The multifunction \eqref{shx} admits a Castaing representation and is Effros measurable.
\end{lemma}
\begin{proof}
Let $\left(c_i\right)_{i\in\mathbb N}$ be a
sequence such that
$\clos{\left\{c_i:i\in\mathbb N\right\}}=\left[\underline c,\overline c\right]$. Then, because
of the continuity of the function $c\to u_c(\omega)$,
\begin{equation*}
  X(\omega)=\clos\big\{u_{c_i}(\omega):i\in\mathbb N\big\}
\end{equation*}
in $\mathcal D'(\mathbb R^{d+1})$, for (almost) all $\omega$.
By the measurability of the $u_{c_i}$ we conclude that
$(u_{c_i})_{i\in\N}$ is a Castaing representation of $X$. By Remark\;\ref{rem1}(b), $X$ is Effros measurable.
\end{proof}
Here is the main result of this section.
\begin{theorem}\label{thm:rs}
The multifunction \eqref{shx} is a random set.
\end{theorem}
\begin{proof}
We wish to apply Theorem\;\ref{thm:fund} with $E = \cD(\R^{n+1})$, $E' = \cD'(\R^{n+1})$. By
\cite[Appendix]{treves}, both $\cD(\R^{n+1})$ and $\cD'(\R^{n+1})$ are Souslin spaces and, in particular, separable. The multifunction
$X$ is Effros measurable by Lemma\;\ref{lem:Effros} and has sequentially compact values. The hypotheses of Theorem\;\ref{thm:fund} are satisfied.
\end{proof}

\section{The One-Dimensional Wave Equation}
\label{sec:1dwave}

In one space dimension, the fundamental solution is given by
\[
   F_c(x,t) = \tfrac1{2c}H(t)H(ct-|x|)
\]
where H denotes the Heaviside function. The convolution of $F_c$ with a function $g$ with support in the upper half-plane results
in d'Alembert's formula
\[
   u_c(x,t) = \tfrac1{2c}\int_0^t\int_{x-cs}^{x+cs}g(y,s) dyds
\]
for the solution of the wave equation \eqref{wave}. This formula can be rewritten as follows. Denote by $\chi_{(x,t,c)}$ the indicator function of the backward triangle
$\{(y,s): 0\leq s\leq t, x-cs \leq y \leq x +cs\}$ with vertex at $(x,t)$. Then $u_c(x,t) = \langle g, \frac1{2c}\chi_{(x,t,c)}\rangle$.

\subsection{Random Field Solution in One Dimension}
\label{subsec:rf1d}

The solution to the one-dimensional ($d=1)$ stochastic wave equation \eqref{stwe1}, \eqref{stwe2} is known \cite{Walsh} to be a scaled Brownian sheet, transformed to the axis
$x = \pm ct$. Actually, by the It\^{o} isometry, the action of $\dot W$ on test functions can be extended to functions in $L^2(\R^2)$. Thus the solution to the one-dimensional stochastic wave equation can be written as
\[
   u_c(x,t) = \langle\dot W, \tfrac1{2c}\chi_{(x,t,c)}\rangle = \tfrac1{2c}\int_0^t\int_{x-cs}^{x+cs} dW(y,s).
\]
\begin{lemma}\label{lem:cov}
Let $(x,t) \in \R\times(0,\infty)$, $c_1, c_2 > 0$. Then the covariance between $u_{c_1}(x,t)$ and $u_{c_2}(x,t)$ is given by
\[
  {\rm E}(u_{c_1}(x,t)u_{c_2}(x,t)) = \frac{t^2}{4}\left(\frac{1}{c_1}\wedge\frac{1}{c_2}\right).
\]
\end{lemma}
\begin{proof}
Again by the It\^{o} isometry,
\begin{align*}
   & {\rm E}(u_{c_1}(x,t)u_{c_2}(x,t))\\& = \tfrac{1}{4c_1c_2}\int_0^\infty\int_{-\infty}^\infty \chi_{(x,t,c_1)}(y,s)\chi_{(x,t,c_2)}(y,s) dyds.
\end{align*}
The latter integral equals the area of the smaller triangle, thus it is $t^2(c_1\wedge c_2)$. After division by $4c_1c_2$, the smaller constant cancels, and this results in the desired formula.
\end{proof}
In particular, $u_c(x,t)$ vanishes as $c\to\infty$ in the mean square sense, i.\,e.,
\[
   \lim_{c\to\infty}{\rm E}(u_{c}^2(x,t)) = 0.
\]
\begin{proposition}
  \label{wie}
  Let $t>0$ and $x\in\R$ be fixed. The stochastic process
  $(v_r)_{r\geq0}$, defined by
  \begin{equation}
    \label{vr}
    v_r=\frac{2}{t}u_{\frac{1}{r}}(t,x), \quad r >0
  \end{equation}
  with $v_0 = 0$ is a Brownian motion.
\end{proposition}
\begin{proof}
  Clearly $v$ is a Gaussian process with zero expectation.
  The covariance we get from Lemma\;\ref{lem:cov} as
  \begin{equation*}
    {\rm E}\left(v_rv_s\right)=\frac{4}{t^2}{\rm E}\left(u_{\frac{1}{r}}(t,x)u_{\frac{1}{s}}(t,x)\right)=r\wedge s.
  \end{equation*}
These properties characterize Brownian motion.
\end{proof}
\begin{proposition}
  Let $\underline c$ and $\overline c$ be two real numbers such that
  \begin{equation*}
    0<\underline c<\overline c<\infty.
  \end{equation*}
  Then, for fixed $t>0$ and $x\in\R$, the multifunction
\begin{align*}
      X:\Omega&\to\cP(\R),\\
      \omega&\to\left\{u_c(x,t,\omega):\underline c\leq c\leq \overline c\right\},
\end{align*}
is a random set.
\end{proposition}
\begin{proof}
  From \eqref{vr} we get
  \begin{equation*}
    u_c(t,x)=\frac{t}{2}v_{\frac{1}{c}}
  \end{equation*}
  almost surely. Since $v$ is a Brownian motion, the function $c\to u_c(x,t,\omega)$ is
  continuous for almost all $\omega$. It follows that the sequence of random variables
  $\left(u_{c_n}(t,x)\right)_{n\in\N}$ is a Castaing representation of
  $X$, where $\left(c_n\right)_{n\in\N}$ is a sequence such that
  $\clos\left\{c_n:n\in\N\right\}=\left[\underline c,\overline c\right]$.
  Applying the Fundamental Measurability Theorem in its classical form for Polish spaces (mentioned in Section\;\ref{sec:meas})
  implies that $X$ is Borel measurable and therefore a random set.
\end{proof}

\subsection{Upper and Lower Probabilities}
\label{subsec:upplow}

In this section, we are going to compute the upper probability $\overline P(B)$ and lower
probability $\underline P(B)$ for any open interval $B=(\underline
b,\overline b)$. By definition, the upper probability equals
\begin{equation}
  \label{upb}
  \overline P(B)=P\left(X\cap B\neq\emptyset\right).
\end{equation}
Applying the theorem of total probability we can write
\begin{align*}
  &P\left(X\cap B\neq\emptyset\right)=P\left(\exists c\in\left[\underline c,\overline c\right]:\underline b<u_c\left(t,x\right)<\overline b\right)\nonumber\\
  &=\int_{-\infty}^{\underline b}P\left(\exists c\in[\underline c,\overline c]:u_c(t,x)>\underline b\big|u_{\overline c}(t,x)=y\right)f(y)dy\nonumber\\
  &+\int_{\underline{b}}^{\overline{b}}\,f(y)dy\\
  &+\int_{\overline b}^\infty P\left(\exists c\in[\underline c,\overline c]:u_c(t,x)<\overline b\big|u_{\overline
      c}(t,x)=y\right)f(y)dy,
\end{align*}
where $f$ is the probability density of the random variable
$u_{\overline c}(t,x)$. Letting $\underline r=1/\overline c$, $\overline r=1/\underline c$,
we get
\begin{align}
  &P\left(X\cap B\neq\emptyset\right)\nonumber \\
  &=\int_{-\infty}^{\frac{2}{t}\underline b}P\left(\exists r\in[\underline r,\overline r]:v_r>\frac{2}{t}\underline b\big|v_{\underline r}=y\right)g(y)dy\nonumber\\
  &+\int_{\frac{2}{t}\underline b}^{\frac{2}{t}\overline b}\,g(y)dy\nonumber\\
  \label{pxcapb}
  &+\int_{\frac{2}{t}\overline b}^\infty P\left(\exists
      r\in[\underline r,\overline r]:v_r<\frac{2}{t}\overline
      b\big|v_{\underline r}=y\right)g(y)dy,
\end{align}
where
\begin{equation*}
  g(y)=\frac{1}{\sqrt{2\pi\underline r}}e^{-\frac{y^2}{2\underline r}}
\end{equation*}
is the probability density of the random variable $v_{\underline r}(t,x)$. It is a Gaussian random variable with variance $\underline
r$. The probabilities in \eqref{pxcapb} we express by first hitting
times of a Brownian motion.

So let $(w_t)_{t\in\left[0,\infty\right)}$ be a standard Brownian motion starting
at $0$ and let $a\in\R$. The first hitting time $\tau(a)$ is defined by
\begin{equation*}
  \tau(a):=\min\left\{t:w_t=a\right\}.
\end{equation*}
Its probability distribution $F_{\tau(a)}$ is well known \cite[Section 7.4]{beichelt}:
\begin{equation*}
  F_{\tau(a)}(t)=P\left(\tau(a)\leq t\right)=\frac{2}{\sqrt{2\pi t}}\int_{|a|}^\infty e^{-\frac{u^2}{2t}}du.
\end{equation*}
If $a>0$, then
\begin{equation*}
  P\left(\exists s\in[0,t]:w_s\geq a\right)=F_{\tau(a)}(t).
\end{equation*}
From the continuity of $F_{\tau(a)}(t)$ with respect to $a$ it follows
that
\begin{align*}
  P\left(\exists s\in[0,t]:w_s>a\right)&=\lim_{\varepsilon\to 0}P\left(\exists s\in[0,t]:w_s\geq a+|\varepsilon|\right)\nonumber\\
  &=\lim_{\varepsilon\to
    0}F_{\tau\left(a+\left|\varepsilon\right|\right)}(t)=F_{\tau(a)}(t).
\end{align*}
Since $v_r$ is a Brownian motion with respect to $r$, it follows that
\begin{align*}
&P\left(\exists r\in[\underline r,\overline r]:v_r>\frac{2}{t}\underline b\big|v_{\underline r}=y\right)\\
   &=F_{\tau\left(\frac{2}{t}\underline b-y\right)}\left(\overline r-\underline r\right)
\end{align*}
and
\begin{align*}
&P\left(\exists r\in[\underline r,\overline r]:v_r<\frac{2}{t}\overline
      b\big|v_{\underline r}=y\right)\\
      &=F_{\tau\left(\frac{2}{t}\overline b-y\right)}\left(\overline r-\underline r\right).
\end{align*}
Hence we can rewrite \eqref{pxcapb} as
\begin{align*}
  P\left(X\cap B\right)\neq\emptyset)&=\int_{-\infty}^{\frac{2}{t}\underline b}F_{\tau\left(\frac{2}{t}\underline b-y\right)}\left(\overline r-\underline r\right)g(y)dy\\
  &+\int_{\frac{2}{t}\underline b}^{\frac{2}{t}\overline b}\,g(y)dy\\
  &+\int_{\frac{2}{t}\overline b}^\infty
  F_{\tau\left(\frac{2}{t}\overline b-y\right)}\left(\overline
    r-\underline r\right)g(y)dy.
\end{align*}
Inserting this into \eqref{upb} shows that the upper
probability $\overline P(B)$ can be expressed in terms of first
hitting times of a Brownian motion.

Finally, we derive a similar
expression for the lower probability
\begin{equation*}
  \underline P(B)=P(X\subset B).
\end{equation*}
Using again the law of total probability gives that
\begin{align*}
\underline P\left(B\right)=P\Big(\tfrac{2}{t}\underline b<v_r<\tfrac{2}{t}\overline b,\forall r\in[\underline r,\overline r]\Big)
\end{align*}
equals
\begin{align}
  \label{lpb}\int_{\tfrac{2}{t}\underline b}^{\frac{2}{t}\overline b}P\Big(\tfrac{2}{t}\underline b<v_r<\frac{2}{t}\overline b,\forall r\in[\underline r,\overline r]\big|v_{\underline r}=y\Big)g(y)dy.
\end{align}
The first exit time $\tau(a,b)$ of a Brownian motion
$\left(w_t\right)_{t\in\left(0,\infty\right)}$ for $a<0$ and $b>0$ is defined by
\begin{equation*}
  \tau(a,b):=\min\left\{t:w_t\notin\left(a,b\right)\right\}.
\end{equation*}
It has the probability distribution
\begin{equation*}
  F_{\tau(a,b)}(t)=\int_0^t\cc_s\left(\frac{b+a}{2},\frac{b-a}{2}\right)ds,
\end{equation*}
where
\begin{equation*}
  \cc_s(x,y)=\mathcal L_{\gamma\to s}^{-1}\left(\frac{\cosh\left(x\sqrt{2\gamma}\right)}{\cosh\left(y\sqrt{2\gamma}\right)}\right)
\end{equation*}
for $x<y$ \cite[p.\,212 and 641]{borodin}. Here $\mathcal
L^{-1}$ denotes the inverse Laplace transform. By the same arguments as in the case of $\overline P$ we can rewrite
\eqref{lpb} as
\begin{equation}
  \underline P(B)=\int_{\frac{2}{t}\underline b}^{\frac{2}{t}\overline b}\left(1-F_{\tau\left(\frac{2}{t}\underline b-y,\frac{2}{t}\overline b-y\right)}\left(\overline r-\underline r\right)\right)g(y)dy.
\end{equation}

\vspace{2mm}
\noindent
{\bf Acknowledgements.}
This work was supported by the FWF doctoral program Computational Interdisciplinary Modelling W1227 as well as
the research project P-27570-N26 of FWF (The Austrian Science Fund).

\end{document}